\documentclass{article}
\usepackage[utf8]{inputenc}
\usepackage{amsmath,amssymb,amsthm,amsfonts}
\usepackage{listings}
\usepackage{fullpage}
\usepackage{tikz}
\usepackage{capt-of}
\usetikzlibrary{positioning,shapes,fit,arrows}
\usepackage{algorithm,algpseudocode}

\usepackage{scalerel}
\usepackage{comment}
\usepackage{titling}
\usepackage{mathrsfs}
\title{Elementary analysis of isolated zeroes of a polynomial system}
\author{%
Mitali Bafna\thanks{School of Engineering and Applied Sciences, Harvard University, Cambridge, Massachusetts, USA. Supported in part by a Simons Investigator Award and NSF Award CCF 1715187. Email: \texttt{mitalibafna@g.harvard.edu}.}
\and Madhu Sudan\thanks{School of Engineering and Applied Sciences, Harvard University, Cambridge, Massachusetts, USA. Supported in part by a Simons Investigator Award and NSF Award CCF 1715187. Email: \texttt{madhu@cs.harvard.edu}.} 
\and Santhoshini Velusamy\thanks{School of Engineering and Applied Sciences, Harvard University, Cambridge, Massachusetts, USA. Supported in part by a Simons Investigator Award and NSF Award CCF 1715187. Email: \texttt{svelusamy@g.harvard.edu}.}
\and David Xiang\thanks{Harvard College. Email: \texttt{davidxiang@college.harvard.edu}.}
}

\newtheorem{theorem}{Theorem}[section]
\newtheorem{lemma}[theorem]{Lemma}
\newtheorem{definition}[theorem]{Definition}
\newtheorem{corollary}[theorem]{Corollary}
\newtheorem{remark}[theorem]{Remark}

\newcommand\numberthis{\addtocounter{equation}{1}\tag{\theequation}}

\usepackage{xcolor}

\newcommand{\f}{\textbf{f}}
\newcommand{\g}{\textbf{g}}
\newcommand{\h}{\textbf{h}}
\newcommand{\va}{\textbf{a}}
\newcommand{\vb}{\textbf{b}}

\newcommand{\F}{\mathbb{F}}
\newcommand{\Z}{\mathbb{Z}}

\begin{document}

\date{}
\maketitle

\begin{abstract}
    Wooley ({\em J. Number Theory}, 1996) gave an elementary proof of a Bezout like theorem allowing one to count the number of isolated integer roots of a system of polynomial equations modulo some prime power. 
    In this article, we adapt the proof to a slightly different setting. Specifically, we consider polynomials with coefficients from a polynomial ring $\mathbb{F}[t]$ for an arbitrary field $\mathbb{F}$ and give an upper bound on the number of isolated roots modulo $t^s$ for an arbitrary positive integer $s$. In particular, using $s=1$, we can bound the number of isolated roots of a system of polynomials over an arbitrary field $\mathbb{F}$. 
\end{abstract}

\section{Introduction}
Wooley~\cite{Wooley} considered a system of $n$ modular polynomial equations in $n$ variables over the integers modulo any prime power $p^s$,
and gave an elementary proof of the fact that the number of {\em non-singular} solutions is bounded by the product of the degrees. 
Later, Zhao~\cite[Lemma A.4]{Zhao}, noted that this result can also be adapted to polynomials over the ring $\F[t]$, the ring of polynomials in variable $t$ over any finite field $\F = \F_q$. 
In this article, whose goal is expository, we give a full proof of this result elaborating on some of the algebraic steps that are omitted in \cite{Zhao} (in particular in the proof of Lemma A.3 there). We note also in passing that the proof does not require the finiteness of the field $\F$. Our proof also simplifies some of the steps mildly. 

 We start with some basic notation and then state the main theorem.
For positive integer $n$, let $[n]$ denote the set $[n] = \{1,\ldots,n\}$. 
Throughout this paper we will be working with the multivariate polynomial ring $R[X_1,\ldots,X_n]$ for $R = \F$ for some field $\F$, or $R = \F[t]$ the ring of polynomials in $t$ over $\F$, or $R = \F(t)$ the field of rational functions in $t$ over $\F$. 
For $f \in R[X_1,\ldots,X_n]$ we let $\deg(f)$ denote the total degree of $f$ in $X_1,\ldots,X_n$  and we let $\deg_{X_i}(f)$ to denote its degree in the variable $X_i$. In particular even when $R = \F[t]$ or $R = \F(t)$ we ignore the degree over $t$ in $\deg(f)$. 
For $f \in \F[t][X_1,\ldots,X_n]$ and $i \in [n]$ let $\frac{\partial f}{\partial X_i}$ denote the partial derivative of $f$ with respect to the variable $X_i$. For a sequence of polynomials $\textbf{f} = (f_1,\ldots,f_n)$, with $f_i \in \F[t][X_1,\ldots,X_n]$, let $$J(\textbf{f}) =  \left[\frac{\partial f_j}{\partial X_i}\right]_{1\leq i,j \leq n}.$$ Note $J(\textbf{f}) \in (\F[t][X_1,\ldots,X_n])^{n \times n}$.  For $f \in \F[t][X_1,\ldots,X_n]$ and $\mathbf{a} \in \F[t]^n$ we let $f(\mathbf{a}) \in \F[t]$ denote the evaluation of $f$ at $\mathbf{a}$.
We use $J(\f;\va)$ to denote the evaluation of $J(\f)$ at $\va$, i.e., each element of $J(\f)$ is evaluated at $\va$. 

\begin{definition}[Isolated Zero]
For a system of polynomials $\f = (f_1,\ldots,f_n) \in (\F[X_1,\ldots,X_n])^n$, we say that $\mathbf{a} \in \F^n$ is an {\em isolated zero} of $\f$ if 
$f_i(\mathbf{a}) = 0$ for every $i \in [n]$ and $\det(J(\f;\mathbf{a})) \ne 0$.
Let $\mathcal{N}(\f)$ denote the number of isolated zeroes of $\f$.

For a system of polynomials, $\f = (f_1,\ldots,f_n) \in (\F[t][X_1,\ldots,X_n])^n$ and a positive integer $s$, we say that $\mathbf{a} \in (\F[t]/t^s)^n$ is an {\em isolated zero} of $\f$ modulo $t^s$ if 
$f_i(\mathbf{a}) = 0 \pmod {t^s}$ for every $i\in[n]$ and $\det(J(\f;\mathbf{a})) \ne 0 (\pmod{t})$.
Let $\mathcal{N}_s(\f)$ denote the number of isolated zeroes of $\f$ modulo $t^s$.
\end{definition}

\begin{theorem}\label{main theorem}
Let $\f = (f_1,\dots,f_n)$ be a sequence of polynomials in $\mathbb{F}[t]\Big[X_1,\dots,X_n\Big]$ with $\deg(f_i) \leq k_i$. 
Then for every positive integer $s$, we have  $\mathcal{N}_s(\textbf{f})\leq k_1\cdots k_n$.
\end{theorem}

From the theorem above, we get the following immediate corollary for counting the isolated zeroes of a system of polynomials over fields. This is one of the main implications of Bezout's theorem.

\begin{corollary}\label{corollary}
Let $\f = (f_1,\dots,f_n)$ be a sequence of polynomials in $\mathbb{F}\Big[X_1,\dots,X_n\Big]$ with $\deg(f_i) \leq k_i$. 
 Then $\mathcal{N}(\textbf{f})\leq k_1\cdots k_n$. 
\end{corollary}

\begin{proof}
The corollary follows by viewing $\f \in \F[X_1,\ldots,X_n]^n$ as elements of $(\F[t][X_1,\ldots,X_n])^n$ and noting that an isolated zero $\mathbf{a} \in \F^n$ of $\f$ is also an isolated zero modulo $t$ of $f$. Thus we have 
$\mathcal{N}(\textbf{f}) \leq \mathcal{N}_1(\textbf{f}) \leq k_1 \cdots k_n$.
\end{proof}

One of the highlights of the proof in \cite{Wooley} is that even though the corollary makes no reference to the variable $t$, the proof works with the ring $\F[t]$ and extensions of it!


\section{Proof of Theorem \ref{main theorem}}

Our proof follows the same outline as that in \cite{Wooley}. We give an overview here.

We consider the setting where $f_1,\ldots,f_n \in \F[X_1,\ldots,X_n]$. Roughly the presence of isolated zeroes implies that $f_1,\ldots,f_n$ can not be algebraically dependent. Our first lemma shows a low-degree algebraic dependence between $f_1,\ldots,f_n$ and the polynomial $X_1$. Specifically we find a polynomial $\Psi(Y_1,\ldots,Y_n,Z)$ of low-degree in $Z$ such that $\Psi(f_1,\ldots,f_n,X_1)$ is the zero polynomial. The intent would be to substitute values from $\F$  for $Y_1,\ldots,Y_n$, say $\alpha_1,\ldots,\alpha_n$,  so that the resulting polynomial $Q(Z) = \Psi(\alpha_1,\ldots,\alpha_n,Z)$ has a zero at $a_1$ for every $\textbf{a} = (a_1,\ldots,a_n)$ that is an isolated zero of $\f = (f_1,\ldots,f_n)$. A natural choice would be to set $Y_i = 0$  and then we do have $\Psi(0,\ldots,0,a_1) = 0$ as desired. This would be interesting if we knew $\Psi(0,\ldots,0,Z)$ is not identically zero, but this is not easy to establish! The central idea in \cite{Wooley} is to set $Y_i$ to some value $\alpha_i \in \F[t]$ such that $\alpha_i \equiv 0 \pmod t$. It is easy to find $\alpha_i$'s satisfying these conditions while ensuring $Q(Z)$ is not identically zero. And we do get 
$Q(a_i) \equiv \Psi(\alpha_1,\ldots,\alpha_n,a_i) \equiv 0 \pmod t$. But it is no longer clear why getting zeroes modulo $t$ of a polynomial that may itself be zero modulo $t$ might be interesting. Here \cite{Wooley} uses a clever idea of lifting the zeroes in $\F[t]$ of $Q(Z)$ modulo $t$, into zeroes from the field $\F((t))$ of $Q(Z)$, where $\F((t))$ is the field of formal Laurent series over $\F$ in the variable $t$. This field contains $\F[t]$ and so $Q(Z)$ becomes a polynomial over this field as well. The lifting itself is not a straightforward application of Hensel lifting, but rather depends on the way $Q(Z)$ was defined and in particular relies on the fact that $\textbf{a}$ is an {\em isolated} zero of $\f$. Putting these ideas together one gets a relatively simple proof of a Bezout-like theorem. 

We start below with a definition of the formal power series ring $\F[[t]]$ and the formal Laurent series field $\F((t))$. We then present our ``algebraic dependence lemma'' showing a low-degree algebraic dependence between $f_1,\ldots,f_n$ and the polynomial $X_1$. (See Lemma~\ref{lemma 2}.) Next, we give a Hensel lifting lemma that shows how to lift isolated zeroes modulo small powers of $t$ from $\F[t]$ of a system of polynomials to a zero from $\F[[t]]$. (See Lemma~\ref{lemma 3}.) Finally we combine these lemmas to obtain a proof of Theorem~\ref{main theorem} at the end of this section.

\begin{definition}[Formal Power Series and Formal Laurent Series]
The formal power series ring $\F[[t]]$ has as its members all formal infinite sums 
$\sum_{i \in \Z_{\geq 0}} a_i t^i$ where $a_i \in \F$. Addition and multiplication are defined in the usual way.
The formal Laurent series ring $\F((t))$ has as its members all formal infinite sums 
$\sum_{i \in \Z} a_i t^i$ where $a_i \in \F$ and the set $\{a_i \ne 0 | i \leq 0\}$ is finite.  Addition and multiplication are defined in the usual way.
\end{definition}

It is well-known that, for every field $\F$, $\F[[t]]$ is an integral domain  and $\F((t))$ is a field. We also use the fact that $\F \subseteq \F[t] \subseteq \F[[t]] \subseteq \F((t))$ where the inclusions preserve the ring operations.

\begin{lemma}\label{lemma 2}
Let $f_1,\dots,f_n$ be polynomials in $\mathbb{F}[t]\Big[X_1,\dots,X_n\Big]$ with respective degrees $k_1,k_2,\dots,k_n$. There exists a non-zero polynomial $\Psi \in \mathbb{F}[t]\Big[Y_1,\dots,Y_n,Z\Big]$ with $\deg_Z(\Psi) \leq  k_1\cdots k_n$ such that $\Psi(f_1,f_2,\dots,f_n,X_1) = 0$.
\end{lemma}

\begin{remark}
We note that in \cite{Wooley} there is an additional requirement that $\deg_{X_1}(\Psi) > 0$. We do not make that a requirement (and this relaxation seems to clean up some steps of the proof below). On the other hand this condition is implied by the other conditions, so we are able to apply our $\Psi$ in the same way as it is in \cite{Wooley}.
\end{remark}

\begin{proof}
In what follows let $B$ and $D$ be two non-negative integers that will be specified later. 

Let $\mathcal{V}_D$ denote the vector space of all polynomials of degree at most $D$ in $\mathbb{F}(t)[X_1,X_2,\dots,X_n]$. The dimension of this space is just the number of monomials of degree at most $D$ which is given by $\binom{D+n}{n}$.

Now consider the following set of monomials,
\begin{equation*}
    M_{B,D} = \left\{Y_1^{d_1}\cdots Y_n^{d_n}Z^r \middle|   d_i,r\in \mathbb{Z}_{\geq 0}, r \leq B, \sum_{i=1}^n d_i k_i + r \leq D\right\}.
\end{equation*} 
Now let $S_{B,D}$ be the evaluations of the elements of $M_{B,D}$ at $Y_i = f_i$ for $i \in [n]$ and $Z = X_1$, i.e., let 
\begin{equation*}
    S_{B,D} = \left\{f_1^{d_1}\cdots f_n^{d_n}X_1^r \middle|   d_i,r\in \mathbb{Z}_{\geq 0}, r \leq B, \sum_{i=1}^n d_i k_i + r \leq D\right\}.
\end{equation*} 
$S_{B,D}$ is thus a subset of $\F(t)[X_1,\ldots,X_n]$. Furthermore every polynomial in $S_{B,D}$ has degree at most $D$ and so $S_{B,D} \subseteq \mathcal{V}_D$. In what follows we show that $|M_{B,D}| > \dim(\mathcal{V}_D)$ (for an appropriate choice of $D$)  to get a non-trivial linear dependence among these monomials when evaluated at $(f_1,\ldots,f_n,X_1)$. This will yield the non-zero polynomial $\Psi$ we are seeking.

For $\textbf{m} = (m_1,\ldots,m_n) \in (\Z_{\geq 0})^n$ and $r \in \Z_{\geq 0}$, denote by $S(r;\textbf{m})$ the number of $n$-tuples $\textbf{d} \in (\Z_{\geq 0})^n$ such that 
\begin{equation*}
    \sum_{i=1}^n k_id_i + m_i \leq D-r ~~~\forall i \in [n].
\end{equation*}
Note that the quantity we wish to lower bound is  $|S_{B,D}| = \sum_{r=0}^B S(r;\mathbf{0})$. Since $S(r;\textbf{m})$ is a non-increasing function of $\textbf{m}$ we also have for every $\textbf{m} \in (\Z_{\geq 0})^n$, $S(r;\textbf{m})\leq S(r;\mathbf{0})$. Using this fact, we get the following inequality,
\begin{align*}
      k_1k_2\cdots k_n S(r;\mathbf{0}) & = \sum_{m_1=0}^{k_1-1}\sum_{m_2=0}^{k_2-1}\dots \sum_{m_n=0}^{k_n-1} S(r;\mathbf{0}) \\
      & \geq \sum_{m_1=0}^{k_1-1}\sum_{m_2=0}^{k_2-1}\dots \sum_{m_n=0}^{k_n-1} S(r;\mathbf{m}) \\
      & = \binom{D+n-r}{n},\numberthis \label{eqn3}
\end{align*}
where the last equality follows from the observation that the previous expression exactly counts the number of solutions of the inequality,
\begin{equation*}
    e_1+e_2+\cdots +e_n \leq D - r, e_i\in \mathbb{Z}_{\geq 0}.
\end{equation*}
Applying Inequality (\ref{eqn3}), we get
\begin{align*}
    \sum_{r=0}^B S(r;0) & \geq \frac{1}{k_1\ldots k_n} \sum_{r=0}^B \binom{D+n-r}{n} \\
    &= \frac{B+1}{k_1\ldots k_n}\binom{D+n}{n}(1-\mathcal{O}(D^{-1})),
\end{align*}
where the constant in the $\mathcal{O}$ notation only depends on $B,\textbf{k},n$ (and does not depend on $D$). It follows that for $B = k_1\dots k_n$, we can choose $D$ sufficiently large so that $(1 + 1/(k_1\cdots k_n))(1 - O(D^{-1})) > 1$ and so 
\begin{equation*}
\sum_{r=0}^R S(r;0) > \binom{D+n}{n} = \dim(\mathcal{V}_d).
\end{equation*}
We thus conclude that $|M_{B,D}| > \dim(\mathcal{V}_d)$ and so there is a non-zero sequence $(C_m \in \F(t))_{m \in M_{B,D}}$ such that for the polynomial $\Psi_0(Y_1,\ldots,Y_n,Z) \triangleq \sum_{m \in M_{B,D}} C_m \cdot m$ we have $\Psi_0(f_1,\ldots,f_n,X_1) = 0$. By construction $\deg_Z(\Psi_0) \leq k_1\cdots k_n$.  By clearing the denominators of the coefficients $C_m$ we can get a polynomial $\Psi \in \F[t]\big[Y_1,\ldots,Y_n,Z\big]$ with the same $Z$-degree which also vanishes at $(f_1,\ldots,f_n,X_1)$. This concludes the proof of the lemma.
\end{proof}




\begin{lemma}\label{lemma 3}
Let $\g = (g_1,\dots,g_n) \in \left(\mathbb{F}[t]\Big[X_1,\dots,X_n\Big]\right)^n$. Suppose that $\textbf{a} \in \mathbb{F}[[t]]^n$ satisfies $g_j(\textbf{a}) \equiv 0 \pmod {t^s}$ for $1\leq j \leq n$ and $\det(J(\textbf{g};\textbf{a})) \neq 0 \mod t$. Then there exists   $\textbf{b}\in \mathbb{F}[[t]]^n$ such that $g_j(\textbf{b}) = 0 $ for all $1\leq j \leq n$ and $\textbf{b}\equiv \textbf{a} \pmod {t^s}$.
\end{lemma}

\begin{proof}
Let $\textbf{a} = \sum_i a(i)t^i$, where $a(i) \in \F^n$. We prove by induction on $i \in \{s,s+1,\ldots\}$ that there exists $\va_i \in \mathbb{F}[[t]]^n$, with $a_s = a$, satisfying:
\begin{enumerate}
    \item $g_j(\va_i) \equiv 0 \mod t^i$ for all $j$,
    \item $\va_i \equiv \va_{i-1} \mod t^{i-1}$, for every $i > s$. 
\end{enumerate}
Note that the case of $i=s$ is true by the hypothesis of the lemma. Assume the inductive assumption is true for some $i\geq s$; we will show how to construct $\va_{i+1}$. We will show that there exists $\vb_i \in \F^n$ such that  $\va_{i+1}\triangleq \va_i+t^i\vb_i$ satisfies
$g_j(\va_{i+1}) \equiv 0 \mod t^{i+1}$ for all $j$. 
By construction, we have that the second condition is satisfied, since $\va_{i+1} = \va_{i} = \va \pmod {t^s}$, hence we only need to find $b_i$ such that the first condition is satisfied.

We first recall how polynomials behave under local perturbations. For every $P \in \F[[t]][X_1,\ldots,X_n]$, there exist polynomials $P^{(k\ell)} \in \F[[t]][X_1,\ldots,X_n,Y_1,\ldots,Y_n]$, for $k,\ell \in [n]$ such that 
$$P(X+Y) = P(X) + \sum_{k\in [n]} \frac{\partial P}{\partial X_k}(X) Y_k + \sum_{k,\ell \in [n]} P^{(k\ell)}(X,Y) Y_k Y_\ell.$$ where $X = (X_1,\ldots,X_n)$ and $Y = (Y_1,\ldots,Y_n)$. Applying it to our functions $g_1,\ldots,g_n$ at $X = \va_i$ and $Y = \vb_i t^i$ we have 
$$ \g(\va_i + t^i \vb_i) = \g(\va_i) + t^i \cdot J(\f;\va_i) \cdot \vb_i + t^{2i} \h(\va_i,\vb_i),$$
for some $\h \in (\F[[t]][X,Y])^n$. 
Thus to get $\g(\va_i + t^i \vb_i) = 0 \mod t^{i+1}$ we need $\g(\va_i) + t^i \cdot J(\f;\va_i)\cdot \vb_i = 0 \mod t^{i+1}$. Since 
$\g(\va_i) = 0 \mod t^i$ we can divide the equation above by $t^i$ to see that we need $J(\f;\va_i)\cdot \vb_i = - t^{-i}\g(\va_i) \mod t$. Note that since we are working modulo $t$, we have that the equation: $(J(\f;\va_i) \mod t) \cdot \vb_i = (-t^{-i}\g(\va_i) \mod t)$, is over field constants. We also have that $\det(J(\g,\va_i)) = \det(J(\g,\va)_ \ne 0 \pmod t$ and so the system above can be inverted, to yield $\vb_i \in \F^n$. This yields $\va_{i+1}$ as desired and completes the inductive step.

The $\vb$ as required by the lemma is now obtained as follows: Let $\vb_i$ for $i \geq s$ be as obtained in the above inductive proof. Let $b(i) = a(i)$, for all $i < s$ and let $b(i) = a(i) + \vb_i$, for all $i \geq s$. Then $\vb =  \sum_{i} b(i) t^i$ gives the element of $\F[[t]]^n$ satisfying $g_j(\vb) = 0$ for all $j \in [n]$ and $\vb \equiv \va \pmod {t^s}$.

\end{proof}

We are now ready to prove Theorem~\ref{main theorem}.

\begin{proof}[Proof of Theorem~\ref{main theorem}:]
Assume for contradiction that the system $\f = (f_1,\ldots,f_n)$ has at least $1+ \prod_{i=1}^n k_i$ isolated zeroes modulo $t^s$. Let $S$ denote such a set of cardinality $1+ \prod_{i=1}^n k_i$. Furthermore assume that the projections of the elements of $S$ to their first coordinate are distinct modulo $t^s$, i.e., for $S_1 \triangleq \{a_1 | \textbf{a} = (a_1,\ldots,a_n) \in S\}$ we assume $|S| = |S_1| = 1+ \prod_{i=1}^n k_i$ and that for $a_1,a'_1 \in S_1$ we have 
$a_1 \ne a'_1 \pmod {t^s}$.
(We can ensure this by applying a generic affine transformation from $\F^n \to \F^n$ to the variables $X_1,\ldots,X_n$ which will ensure that the first coordinates are all distinct. If the field $\F$ is finite we can apply this transformation over any large extension field.) 

Let $\Psi \in \mathbb{F}[t][Y_1,Y_2,\dots,Y_n,Z]$ denote the non-zero polynomial given by  applying Lemma \ref{lemma 2} to the system $\f = (f_1,\ldots,f_n)$. So $\Psi$ satisfies $\deg_Z(\Psi) \leq \prod_{i=1}^nk_i$ and $\Psi(f_1,f_2,\dots,f_n,X_1) = 0$. Since $\Psi$ is a non-zero polynomial, there always exists constants $c_1,c_2,\dots c_n \in \mathbb{F}$ such that the polynomial, $Q(Z)\triangleq \Psi(c_1t^s,c_2t^s,\dots,c_nt^s,Z)\in \mathbb{F}[t][Z]$ is a non-zero polynomial of degree at most $\prod_{i=1}^n k_i$. We will first show that the elements of $S_1$ are zeros of $Q \mod t^s$. We will then be able to lift these to distinct zeroes of $Q$ in the ring $\F[[t]]$ and hence in the field $\F((t))$. This will yield the desired contradiction since $Q$ is a non-zero univariate polynomial with degree at most $\prod k_i$, whereas we will show that it has $|S_1|$ distinct roots, which is more roots than its degree.

Consider an isolated zero modulo $t^s$, say $\textbf{a} = (a_1,a_2,\dots,a_n) \in S$,  of the system $\f$. We have \begin{align*}
    Q(a_1) & \equiv \Psi (c_1 t^s, ... , c_d t^s, a_1) \pmod {t^s}\\
    & \equiv \Psi (0, ... , 0, a_1) \pmod {t^s}~~~\mbox{(Since $c_i t^s \pmod {t^s} \equiv 0$)}\\
    & \equiv \Psi (f_1(\textbf{a}),\dots,f_n(\textbf{a}),a_1) \pmod {t^s}~~~\mbox{(Since $f_i(\textbf{a}) \pmod {t^s} \equiv 0$)} \\
    & \equiv 0 \pmod {t^s}.\\
\end{align*}
Next we show that every such zero $a_1 \in \F[t]$ of $Q  \mod t^s$ can be lifted to a distinct zero of $Q$ in the ring $\F[[t]]$. 

We first note that if there exists $\textbf{b}=(b_1,b_2,\dots,b_n) \in \mathbb{F}[[t]]^n$ such that for all $i$, $f_i(\textbf{b})=c_i t^s$, then we have that $Q(b_1) = \Psi(c_1 t^s,\ldots,c_n t^s, b_1) = \Psi (f_1(\textbf{b}),\ldots,f_n(\textbf{b}),b_1) = 0$. Thus to get our lifted $b_1$ it suffices to find a $b_1$ such that $b_1 = a_1 \pmod {t^s}$ and $f_i(\textbf{b})=c_i t^s,\forall i$. Now let $g_i(X_1,\ldots,X_n) = f_i(X_1,\ldots,X_n) - c_i t^s$, for $i \in [n]$. By construction we have
that $\textbf{a}$ is an isolated zero of $\g = (g_1,\ldots,g_n)$ modulo $t^s$ (the shifts by $c_i t^s$ do not alter this). Thus by applying Lemma~\ref{lemma 3} to $\g$ and $\textbf{a}$ we get that there exists $\textbf{b} \in \F[[t]]^n$ that is a zero of $\g$, and thus satisfies
$f_i(\textbf{b}) = c_i t^s$ for every $i \in [n]$. Thus, we conclude that $Q(b_1) = 0$. We also have that for distinct $a_1$ and $a'_1 \in S_1$ their lifts $b_1$ and $b'_1$ satisfy $b_1 \equiv a_1 \not\equiv a'_1 \equiv b'_1 \pmod t$, so $b_1$ and $b'_1$ are distinct. So, we conclude that $Q$ has more zeroes than its degree yielding the desired contradiction with the assumption that $\f$ has $1 + \prod_{i=1}^n k_i$ isolated zeroes. 
\end{proof}

\section{Further discussion}

There are two main differences between Theorem~\ref{main theorem} and the well-known Bezout's theorem: (1) The usual theorem can count isolated zeroes with multiplicities (and in particular a mutliple zero is counted as several isolated zeroes), whereas we do not count such zeroes. (2) The usual theorem allows more polynomials than the number of variables and this is particularly helpful in counting the number of zeroes on subspaces even when the polynomials have algebraic dependence. It would be interesting to see if the methods from \cite{Wooley} and this paper can be extended to either of these cases.

\section*{Acknowledgements}

We are grateful to Trevor Wooley for pointing us to the results in \cite{Zhao}.

\end{document}